\documentclass[a4paper,11pt]{amsart}
\usepackage{amsmath, amsfonts, amssymb,amscd,graphics,graphicx,color,eucal} 
\usepackage{latexsym}\def\qed{\hfill$\Box$}   

\numberwithin{equation}{section}

\theoremstyle{plain}
\newtheorem{thm}{Theorem}[section]
\newtheorem{lem}[thm]{Lemma} 
\newtheorem{cor}[thm]{Corollary}
\newtheorem{prop} [thm] {Proposition} 

\theoremstyle{definition}
\newtheorem*{defn}{Definition}
\newtheorem*{rmk}{Remark}
\newtheorem*{ex}{Example}

\begin{document}
  
\title[Cohomology rings of $(n-k,k)$ Springer varieties]{The $S^1$-equivariant cohomology rings of $(n-k,k)$ Springer varieties}
\author{Tatsuya Horiguchi}
\address{Department of Mathematics, Osaka City University, Sumiyoshi-ku, Osaka 558-8585, Japan}
\email{d13saR0z06@ex.media.osaka-cu.ac.jp}
\date{\today}
\maketitle

\begin{abstract}
The main result of this note gives an explicit presentation of the $S^1$-equivariant cohomology ring of the $(n-k,k)$ Springer variety (in type $A$) as a quotient of a polynomial ring by an ideal $I$, in the spirit of the well-known Borel presentation of the cohomology of the flag variety. 
\end{abstract}

\setcounter{tocdepth}{1}
\tableofcontents

\section{Introduction} \label{sect:1}

The Springer variety $\mathcal{S}_N$ associated to a nilpotent operator $N\colon \mathbb{C}^n\to \mathbb{C}^n$ is the subvariety of $Flags(\mathbb{C}^n)$ defined as 
\begin{center}
$\mathcal{S}_N=\{V_{\bullet} \in Flags(\mathbb{C}^n)\mid NV_i\subseteq V_{i-1} \ $for all$ \ 1\leq i\leq n \}$
\end{center}
where $V_{\bullet}$ denotes a nested sequence 
$$0=V_0 \subset  V_1 \subset  \dots \subset  V_{n-1} \subset  V_n=\mathbb{C}^n$$ 
of subspaces of $\mathbb{C}^n$ and $\dim_{\mathbb{C}}V_i=i$ for all $i$. When $N$ consists of two Jordan blocks of sizes $n-k$ and $k$ with $n\ge 2k$, we denote $\mathcal{S}_N$ by $\mathcal{S}_{(n-k,k)}$. 
The cohomology ring of Springer variety $\mathcal{S}_N$ has been much studied due to its relation to representations of the permutation group on $n$ letters (\cite{spr1}, \cite{spr2}).  In fact, the ordinary cohomology ring $H^*(\mathcal{S}_N;\mathbb{Q})$ is known to be the quotient of a polynomial ring by an ideal called Tanisaki's ideal (\cite{t}). In this paper we study the equivariant cohomology ring of $\mathcal{S}_{(n-k,k)}$ with respect to a certain circle action on $\mathcal{S}_N$ which we describe below.

Recall that the $n$-dimensional compact torus $T$ consisting of diagonal unitary matrices of size $n$ acts on $Flags(\mathbb{C}^n)$ in a natural way. A certain circle subgroup $S$ of $T$ leaves $\mathcal{S}_N$ invariant (cf. Section~\ref{sect:2}). The ring homomorphism 
$$H^{\ast}_{T}(Flags(\mathbb{C}^n);\mathbb{Q})\rightarrow H^{\ast}_{S}(\mathcal{S}_N;\mathbb{Q})$$
induced from the inclusions of $\mathcal{S}_N$ into $Flags(\mathbb{C}^n)$ and $S$ into $T$ is known to be surjective (cf. \cite{h-s}). The main result of this paper is an explicit presentation of $H^{\ast}_{S}(\mathcal{S}_{(n-k,k)};\mathbb{Q})$ as a ring using the epimorphism above (Theorem~\ref{theo:3.1}).   
In related work, Dewitt and Harada \cite{d-h} give a module basis of $H^*_S(\mathcal{S}_{(n-k,k)};\mathbb{Q})$ over $H^*(BS;\mathbb{Q})$ when $k=2$ from the viewpoint of Schubert calculus.  

Finally, since the restriction map 
\[
H^{\ast}_{S}(\mathcal{S}_{N};\mathbb{Q})\to H^{\ast}(\mathcal{S}_{N};\mathbb{Q})
\]
is also known to be surjective for any nilpotent operator $N$, our presentation of $H^{\ast}_{S}(\mathcal{S}_{(n-k,k)};\mathbb{Q})$ yields a presentation of  $H^{\ast}(\mathcal{S}_{(n-k,k)};\mathbb{Q})$ as a ring (Corollary \ref{coro:3.1}). However, the resulting presentation is slightly different from the one given in \cite{t}.

This paper is organized as follows. We briefly recall the necessary background in Section~\ref{sect:2}. Our main theorem, Theorem~\ref{theo:3.1}, is formulated in Section~\ref{sect:3} and proved in Section~\ref{sect:4}.

\medskip
\noindent
\textbf{Acknowledgements.}
The author thanks Professor Mikiya Masuda for valuable discussions and useful suggestions, and thanks Professor Megumi Harada for useful suggestions, and thanks Yukiko Fukukawa for valuable discussions.

\section{Nilpotent Springer varieties and $S^1$-fixed points} \label{sect:2}

We begin by recalling the definition of the nilpotent Springer varieties in type A. 
Since we work exclusively with type A in this paper, we henceforth omit it from our terminology.

The flag variety $Flags(\mathbb{C}^n)$ is the projective variety of nested subspaces in $\mathbb{C}^n$, i.e.
$$Flags(\mathbb{C}^n)=\{V_{\bullet}=(0 = V_0 \subset  V_1 \subset  \dots \subset  V_{n-1} \subset  V_n=\mathbb{C}^n) \mid \dim_{\mathbb{C}}V_i=i \}.$$

\begin{defn} 
Let $N\colon \mathbb{C}^n\to \mathbb{C}^n$ be a nilpotent operator. The \textbf{(nilpotent) Springer variety} $\mathcal{S}_N$ associated to $N$ is defined as 
\[
\text{$\mathcal{S}_N=\{V_{\bullet} \in Flags(\mathbb{C}^n)\mid NV_i\subseteq V_{i-1} \ $for all$ \ 1\leq i\leq n \}$.}
\]
\end{defn}

Since $\mathcal{S}_{gNg^{-1}}$ is homeomorphic (in fact, isomorphic as algebraic varieties) to $\mathcal{S}_N$ for any $g\in GL_n(\mathbb{C})$, we may assume that $N$ is in Jordan canonical form with Jordan blocks of weakly decreasing sizes. Let $\lambda _N$ denote the partition of $n$ with entries the sizes of the Jordan blocks of $N$.  
The $n$-dimensional torus $T$ consisting of diagonal unitary matrices of size $n$ acts on $Flags(\mathbb{C}^n)$ in a natural way and the circle subgroup $S$ of $T$ defined as 
\begin{equation} \label{eq:S}
S=\left\{ \begin{pmatrix}
  g  &    &    &     \\ 
    &  g^{2}  &    &         \\
    &    &  \ddots  &         \\
    &    &    &      g^n  
\end{pmatrix} \mid  \; g\in\mathbb{C}, |g|=1 \right\}
\end{equation}
leaves $\mathcal{S}_N\subseteq Flags(\mathbb{C}^n)$ invariant (see \cite{h-t}). 
The $T$-fixed point set $Flags(\mathbb{C}^n)^{T}$ of $Flags(\mathbb{C}^n)$ is given by 
$$\{(\langle e_{w(1)}\rangle \subset \langle e_{w(1)},e_{w(2)}\rangle \subset \dots \subset \langle e_{w(1)},e_{w(2)},\dots ,e_{w(n)}\rangle =\mathbb{C}^n) \mid w\in S_n\}$$
where $e_1,e_2,\dots ,e_n$ is the standard basis of $\mathbb{C}^n$ and $S_n$ is the permutation group on $n$ letters $\{1,2,\dots ,n \}$, 
so we identify $Flags(\mathbb{C}^n)^{T}$ with $S_n$ as is standard. Also, since the $S$-fixed point set $Flags(\mathbb{C}^n)^{S}$ of  $Flags(\mathbb{C}^n)$ agrees with $Flags(\mathbb{C}^n)^{T}$, we have 
$$\mathcal{S}_N^{S}=\mathcal{S}_N\cap Flags(\mathbb{C}^n)^{S}=\mathcal{S}_N\cap Flags(\mathbb{C}^n)^{T}\subset S_n.$$  

We denote by $\mathcal{S}_{(n-k,k)}$ the Springer variety corresponding to the partition $\lambda _N=(n-k,k)$  with $2k\leq n$.
We next describe the $S$-fixed points in $\mathcal{S}_{(n-k,k)}$. Let $w_{\ell_1,\ell_2,\dots,\ell_k}$ be an element of $S_n$ defined by
\begin{align} \label{eq:2.1}
w_{\ell_1,\ell_2,\dots,\ell_k}(i)=\begin{cases}
n-k+j       \ \ \ \ \ \  $if$ \ i=\ell_j,  \\
i-j \ \ \ \ \ \ \ \ \ \ \ \ $if$ \ \ell_{j}<i<\ell_{j+1}, 
\end{cases} 
\end{align}
where $\ell_0:=0$, $\ell_{k+1}:=n+1$.  Note that $w_{\ell_1,\ell_2,\dots,\ell_k}^{-1}(i)<w_{\ell_1,\ell_2,\dots,\ell_k}^{-1}(i')$ if $1\le i<i'\le n-k$ or $n-k+1\le i<i'\le n$.  

\begin{ex}
Take $n=4$ and $k=2$.  Using one-line notation, the set of permutations of the form described in \eqref{eq:2.1} are as follows:
$$[3,4,1,2],[3,1,4,2],[3,1,2,4],[1,3,4,2],[1,3,2,4],[1,2,3,4].$$
\end{ex}

\begin{lem}\label{lem2.1} 
The $S$-fixed points $\mathcal{S}_{(n-k,k)}^{S}$ of the Springer variety $\mathcal{S}_{(n-k,k)}$ is the set
$$\{w_{\ell_1,\ell_2,\dots,\ell_k}\in S_n \mid 1\leq \ell_1<\ell_2<\dots <\ell_k\leq n \}.$$
\end{lem}

\begin{proof}
Since $\mathcal{S}_{(n-k,k)}^{S}\subset Flags(\mathbb{C}^n)^{T}$, any element $V_{\bullet}$ of $\mathcal{S}_{(n-k,k)}^S$ is of the form 
\begin{align*}
V_{\bullet}=(\langle e_{w(1)}\rangle \subset \langle e_{w(1)},e_{w(2)}\rangle \subset \dots \subset \langle e_{w(1)},e_{w(2)},\dots ,e_{w(n)}\rangle )
\end{align*}
for some $w\in S_n$. Since $N$ is the nilpotent operator consisting of two Jordan blocks with weakly decreasing sizes $(n-k,k)$,
$$Ne_{i}=\begin{cases}
0      \quad &\text{if $i=1 $ or $ n-k+1$},   \\
e_{i-1} \quad&\text{otherwise}. 
\end{cases} $$
Therefore, if $V_{\bullet}$ belongs to $\mathcal{S}_{(n-k,k)}$, then $w(1)=1$ or $n-k+1$. If $w(1)=1$ then $w(2)=2$ or $n-k+1$. 
If $w(1)=n-k+1$ then $w(2)=1$ or $n-k+2$, and so on.  
This shows that $w=w_{\ell_1,\ell_2,\dots,\ell_k}$ for some $1\leq \ell_1<\ell_2<\dots <\ell_k\leq n$. 
Conversely, one can easily see that $w_{\ell_1,\ell_2,\dots,\ell_k}\in \mathcal{S}_{(n-k,k)}^S$.
\end{proof}

\section{Main theorem} \label{sect:3}

In this section, we formulate our main theorem which gives an explicit presentation of the $S$-equivariant cohomology ring of the $(n-k,k)$ Springer variety.

First, we recall an explicit presentation of the $T$-equivariant cohomology ring of the flag variety. 
Let $E_i$ be the subbundle of the trivial vector bundle $Flags(\mathbb{C}^n)\times \mathbb{C}^n$ over $Flags(\mathbb{C}^n)$ whose fiber at a flag $V_{\bullet}$ is just $V_i$. We denote the $T$-equivariant first Chern class of the line bundle $E_i/E_{i-1}$ by $\bar x_i\in H^2_T(Flags(\mathbb{C}^n);\mathbb{Q})$. 
The torus $T$ consisting of diagonal unitary matrices of size $n$ has a natural product decomposition $T\cong (S^1)^n$ where $S^1$ is the unit circle of $\mathbb{C}$. This decomposition identifies $BT$ with $(BS^1)^n$ and induces an identification 
$$H^{\ast}_{T}(pt;\mathbb{Q})=H^{\ast}(BT;\mathbb{Q})\cong \bigotimes H^{\ast}(BS^1;\mathbb{Q})\cong \mathbb{Q}[t_1,\dots,t_n],$$
where $t_i \ (1\leq i\leq n)$ denotes the element corresponding to a fixed generator $t$ of $H^{2}(BS^1;\mathbb{Q})$.
Then $H^{\ast}_{T}(Flags(\mathbb{C}^n);\mathbb{Q})$ is generated by $\bar x_1,\dots ,\bar x_n,t_1,\dots ,t_n$ as a ring. We define a ring homomorphism $\pi $ from the polynomial ring $\mathbb{Q}[x_1,\dots ,x_n]$ to $H^{\ast}_{T}(Flags(\mathbb{C}^n);\mathbb{Q})$ by $\pi (x_i)=\bar x_i$. It is known that $\pi$ is an epimorphism and Ker\hspace{1.5pt}$\pi $ is generated as an ideal by $e_i(x_1,\dots ,x_n)-e_i(t_1,\dots ,t_n)$ for all $1\leq i\leq n$, where $e_i$ is the $i$th elementary symmetric polynomial. Thus, we have an isomorphism:
$$H^{\ast}_{T}(Flags(\mathbb{C}^n);\mathbb{Q})\cong \mathbb{Q}[x_1,\dots,x_n,t_1,\dots,t_n]/(e_i(x_1,\dots ,x_n)-e_i(t_1,\dots ,t_n),1\leq i\leq n).$$

We consider the following commutative diagram:
\begin{equation}\label{eq:3.1}
\begin{CD}
H^{\ast}_{T}(Flags(\mathbb{C}^n);\mathbb{Q})@>{\iota _1}>> H^{\ast}_{T}(Flags(\mathbb{C}^n)^{T};\mathbb{Q})=\displaystyle \bigoplus_{w\in S_n} \mathbb{Q}[t_1,\dots,t_n]\\
@V{\pi _1}VV @V{\pi _2}VV\\
H^{\ast}_{S}(\mathcal{S}_N;\mathbb{Q})@>{\iota _2}>> H^{\ast}_{S}(\mathcal{S}_N^{S};\mathbb{Q})=\displaystyle \bigoplus_{w\in \mathcal{S}_N^{S}\subset S_n} \mathbb{Q}[t]
\end{CD}
\end{equation}
where all the maps are induced from inclusion maps, and we have an identification
$$H^{\ast}_{S}(pt;\mathbb{Q})=H^{\ast}(BS;\mathbb{Q})\cong H^{\ast}(BS^1;\mathbb{Q})\cong \mathbb{Q}[t]$$
where we identify $S$ with $S^1$ through the map diag($g,g^{2},\dots, g^n$) $\mapsto g$.
The maps $\iota _1$ and $\iota _2$ in \eqref{eq:3.1} are injective since the odd degree cohomology groups of $Flags(\mathbb{C}^n)$ and $\mathcal{S}_N$ vanish. The map $\pi _1$ in \eqref{eq:3.1} is known to be surjective (cf. \cite{h-s}) and the map $\pi _2$ is obviously surjective. Since $\pi _1$ is surjective, we have the following lemma. Let $\tau _i$ be the image $\pi _1(\bar x_i)$ of $\bar x_i$ for each $i$.

\begin{lem}\label{lem3.0} 
The $S$-equivariant cohomology ring $H^{\ast}_{S}(\mathcal{S}_N;\mathbb{Q})$ is generated by $\tau _1,\dots, \tau _n, t$ as a ring where $\tau _i$ is the image of $\bar x_i$ under the map $\pi _1$ in \eqref{eq:3.1}. 
\end{lem}

We next consider relations between $\tau _1,\dots, \tau _n$, and $t$. We have  
$$\iota_2(\tau _i)|_w=w(i)t$$
because $\iota _1(\bar x_i)|_w=t_{w(i)}$, $\iota _1(t_i)|_w=t_i$, and $\pi _2(t_i)=it$, where $f|_w$ denotes the $w$-component of $f\in \displaystyle \bigoplus_{w\in S_n} \mathbb{Q}[t_1,\dots,t_n]$.

\begin{lem}\label{lem3.1} 
The elements $\tau _1,\dots ,\tau _n,t$ satisfy the following relations:
\begin{eqnarray}
&\displaystyle\sum_{1\leq i\leq n} \tau_i-\frac{n(n+1)}{2}t=0, & \label{eq:3.2}\\
&(\tau_i+\tau_{i-1}-(n-k+i)t)(\tau_i-\tau_{i-1}-t)=0 \qquad  (1\leq i\leq n), & \label{eq:3.3}\\
&\displaystyle\prod_{0\leq j\leq k}(\tau_{i_j}-(i_j-j)t))=0  \ \ \ \ \ \ \ \ \ \ \ \ (1\leq i_0<\dots <i_{k}\leq n) & \label{eq:3.4}
\end{eqnarray}
where $\tau _{0}=0$.
\end{lem}

\begin{proof}
The relation \eqref{eq:3.2} follows from a relation in $H^*_T(Flags(\mathbb{C}^n);\mathbb{Q})$.  In fact, 
$$\displaystyle\sum_{1\leq i\leq n} \tau_i-\frac{n(n+1)}{2}t=\pi _1((e_1(\bar x_1,\dots ,\bar x_n)-e_1(t_1,\dots ,t_n)))=0.$$

In the following, we denote $\iota _2(\tau _i)$ by the same notation $\tau _i$ for each $i$. To prove the relation \eqref{eq:3.3}, it is sufficient to prove  either 
\begin{equation} \label{eq:or}
\text{$(\tau_i+\tau_{i-1}-(n-k+i)t)|_{w_{\ell_1,\ell_2,\dots,\ell_k}}=0$ \ or \  $(\tau_i-\tau_{i-1}-t)|_{w_{\ell_1,\ell_2,\dots,\ell_k}}=0$}
\end{equation}
for any $w_{\ell_1,\ell_2,\dots,\ell_k}\in \mathcal{S}_{(n-k,k)}^{S}$ since the restriction map $\iota_2$ in \eqref{eq:3.1} is injective. 

We first treat the case $i=1$. By the definition of $w_{\ell_1,\ell_2,\dots,\ell_k}$ in \eqref{eq:2.1} the following holds: \\
$$\tau_1|_{w_{\ell_1,\ell_2,\dots,\ell_k}}=w_{\ell_1,\ell_2,\dots,\ell_k}(1)t=\begin{cases}
(n-k+1)t \ \ \ \ \ \ $if$ \ \ell_1=1,  \\
t       \ \ \ \ \ \ \ \ \ \ \ \ \ \ \ \ \ \ \ \ \  $if$ \ \ell_1\neq 1. 
\end{cases} $$
This shows \eqref{eq:or} for $i=1$ because $\tau _{0}=0$.

We now treat the case $1<i\leq n$. Note that
\begin{align}
(\tau_i-\tau_{i-1})|_{w_{\ell_1,\ell_2,\dots,\ell_k}}&=(w_{\ell_1,\ell_2,\dots,\ell_k}(i)-w_{\ell_1,\ell_2,\dots,\ell_k}(i-1))t, \label{eq:3.5} \\
(\tau_i+\tau_{i-1})|_{w_{\ell_1,\ell_2,\dots,\ell_k}}&=(w_{\ell_1,\ell_2,\dots,\ell_k}(i)+w_{\ell_1,\ell_2,\dots,\ell_k}(i-1))t. \label{eq:3.6}
\end{align}
We take four cases depending on whether $i-1$ and $i$ appear in $\ell_1,\dots,\ell_k$ or not.  

(i) If $\ell_{j}=i-1<i=\ell_{j+1}$ for some $1\le j\le k-1$, then by \eqref{eq:2.1} and \eqref{eq:3.5}, 
$$(\tau_i-\tau_{i-1})|_{w_{\ell_1,\ell_2,\dots,\ell_k}}=((n-k+j+1)-(n-k+j))t=t.$$

(ii) If $\ell_{j}<i-1<i<\ell_{j+1}$ for some $0\le j\le k$, then by \eqref{eq:2.1} and \eqref{eq:3.5}, 
$$(\tau_i-\tau_{i-1})|_{w_{\ell_1,\ell_2,\dots,\ell_k}}=((i-j)-(i-j-1))t=t.$$

(iii) If $\ell_j=i-1<i<\ell_{j+1}$ for some $1\le j\le k$, then by \eqref{eq:2.1} and \eqref{eq:3.6},  
$$(\tau_i+\tau_{i-1})|_{w_{\ell_1,\ell_2,\dots,\ell_k}}=((i-j)+(n-k+j))t=(n-k+i)t.$$

(iv) If $\ell_{j-1}<i-1<i=\ell_j$ for some $1\le j\le k$, then by \eqref{eq:2.1} and \eqref{eq:3.6},  
$$(\tau_i+\tau_{i-1})|_{w_{\ell_1,\ell_2,\dots,\ell_k}}=((n-k+j)+(i-j))t=(n-k+i)t.$$
Therefore, \eqref{eq:or} holds in all cases, proving the relations \eqref{eq:3.3}. 

Finally we prove the relations \eqref{eq:3.4}. For any $w_{\ell_1,\ell_2,\dots,\ell_k}\in \mathcal{S}_{(n-k,k)}^{S}$, there is a positive integer $i_j$ such that $\ell_{j}<i_j<\ell_{j+1}$ for some $0\le j\le k$.  Thus, we have 
$$w_{\ell_1,\ell_2,\dots,\ell_k}(i_j)=i_j-j.$$
This means that 
$$\displaystyle\prod_{0\leq j\leq k}(\tau_{i_j}-(i_j-j)t)|_{w_{\ell_1,\ell_2,\dots,\ell_k}}=0.$$
Therefore, the relations \eqref{eq:3.4} hold, and the proof is complete. 
\end{proof}

It follows from Lemma~\ref{lem3.1} that we obtain a well-defined ring homomorphism 
\begin{align} \label{eq:3.7}
\varphi :\mathbb{Q}[x_1,\dots, x_n,t]/I\rightarrow H^{\ast}_{S}(\mathcal{S}_{(n-k,k)};\mathbb{Q})
\end{align}
where $I$ is the ideal of a polynomial ring  $\mathbb{Q}[x_1,\dots, x_n,t]$ generated by the following three types of elements:
\begin{eqnarray}
&\displaystyle\sum_{1\leq i\leq n} x_i-\frac{n(n+1)}{2}t, & \label{eq:3.8} \\
&(x_i+x_{i-1}-(n-k+i)t)(x_i-x_{i-1}-t) \qquad  (1\leq i\leq n), & \label{eq:3.9}\\
&\displaystyle\prod_{0\leq j\leq k}(x_{i_j}-(i_j-j)t)  \ \ \ \ \ \ \ \ \ \ \ \ \ \ \ (1\leq i_0<\dots <i_{k}\leq n) & \label{eq:3.10}
\end{eqnarray}
where $x_{0}=0$. Moreover, $\varphi $ is surjective by Lemma~\ref{lem3.0}.

The following is our main theorem and will be proved in the next section.

\begin{thm} \label{theo:3.1}
Let $\mathcal{S}_{(n-k,k)}$ be the $(n-k,k)$ Springer variety with $0\leq k\leq n/2$ and let the circle group $S$ act on $\mathcal{S}_{(n-k,k)}$ as described in Section~\ref{sect:2}. Then the $S$-equivariant cohomology ring of $\mathcal{S}_{(n-k,k)}$ is given by 
$$H^{\ast}_{S}(\mathcal{S}_{(n-k,k)};\mathbb{Q})\cong \mathbb{Q}[x_1,\dots,x_n,t]/I$$
where $H^*_{S}(pt;\mathbb{Q})=\mathbb{Q}[t]$ and 
$I$ is the ideal of the polynomial ring  $\mathbb{Q}[x_1,\dots, x_n,t]$ generated by the elements listed in \eqref{eq:3.8}, \eqref{eq:3.9}, and \eqref{eq:3.10}.
\end{thm}\  
 
Since the ordinary cohomology ring of $\mathcal{S}_{(n-k,k)}$ can be obtained by taking $t=0$ in Theorem~\ref{theo:3.1}, we obtain the following corollary. 

\begin{cor} \label{coro:3.1}
Let $\mathcal{S}_{(n-k,k)}$ be $(n-k,k)$ Springer variety with $0\leq k\leq n/2$. Then the ordinary cohomology ring of $\mathcal{S}_{(n-k,k)}$ is given by 
$$H^{\ast}(\mathcal{S}_{(n-k,k)};\mathbb{Q})\cong \mathbb{Q}[x_1,\dots,x_n]/J$$
where $J$ is the ideal of the polynomial ring $\mathbb{Q}[x_1,\dots,x_n]$ generated by the following three types of elements:
\begin{eqnarray*}
&\displaystyle\sum_{1\leq i\leq n} x_i , & \\
&x_i^2 & \qquad  (1\leq i\leq n),\\
&\displaystyle\prod_{1\leq j\leq k+1}x_{i_j}&  \qquad (1\leq i_1<\dots <i_{k+1}\leq n).
\end{eqnarray*}
\end{cor}  

\begin{rmk} 
A ring presentation of the cohomology ring of the Springer variety $\mathcal{S}_N$ is given in \cite{t} for an arbitrary nilpotent operator $N$. Specifically, it is the quotient of a polynomial ring by an ideal called Tanisaki's ideal.  When $\lambda_N=(n-k,k)$, Tanisaki's ideal is generated by the following three types of elements:
\begin{eqnarray*}
&e_1(x_1,\dots ,x_n) , & \\
&e_2(x_{i_1},\dots ,x_{i_{n-1}})&  \qquad (1\leq i_1<\dots <i_{n-1}\leq n),\\
&e_{k+1}(x_{i_1},\dots ,x_{i_{k+1}})&  \qquad (1\leq i_1<\dots <i_{k+1}\leq n),
\end{eqnarray*}
where $e_i$ is the $i$th elementary symmetric polynomial. Note that the first and third elements above are the same as those in Corollary~\ref{coro:3.1}. In fact, one can easily check that Tanisaki's ideal above agrees with the ideal $J$ in Corollary~\ref{coro:3.1} although the generators are slightly different.
\end{rmk}

\section{Proof of the main theorem} \label{sect:4}

This section is devoted to the proof of Theorem~\ref{theo:3.1}. More precisely, we will prove that the epimorphism $\varphi $ in \eqref{eq:3.7} is an isomorphism.
For this, we first find generators of $\mathbb{Q}[x_1,\dots, x_n,t]/I$ as a $\mathbb{Q}[t]$-module.

Recall that a \textbf{filling} of $\lambda $ by the alphabet $\{1,\dots ,n \}$ is an injective placing of the integers $\{1,\dots ,n \}$ into the boxes of $\lambda $. 

\begin{defn} 
Let $\lambda $ be a Young diagram with $n$ boxes. A filling of $\lambda $ is a \textbf{permissible filling} if for every horizontal adjacency \begin{picture}(20,10)                                  
                      \put(0,-2){\framebox(10,10){$a$}}
                      \put(10,-2){\framebox(10,10){$b$}}                      
\end{picture} we have $a<b$.
Also, a permissible filling is a \textbf{standard tableau} if for every vertical adjacency \begin{picture}(10,20)                                  
                      \put(0,3){\framebox(10,10){$a$}}
                      \put(0,-7){\framebox(10,10){$b$}}                      
\end{picture} we have $a<b$.
\end{defn}

Let $T$ be a permissible filling of $(n-\ell,\ell)$ with $0\leq \ell\leq k$. Let $j_1,j_2,\dots ,j_\ell$ be the numbers in the bottom row of $T$. We define $x_{T}:=x_{j_1}x_{j_2}\dots x_{j_\ell}$ and $x_{T_0}:=1$ where $T_0$ is the standard tableau on $(n)$.

\begin{prop} \label{prop4.1} 
The set $\{x_{T} \mid T$ standard tableau on $ (n-\ell,\ell) $ with $ 0\leq \ell \leq k \} $ generates $\mathbb{Q}[x_1,\dots,x_n,t]/I$ as a $\mathbb{Q}[t]$-module.
\end{prop}

\begin{proof} It is sufficient to prove that $x_{b_1}x_{b_2}\cdots x_{b_\ell}$ \ ($1\leq {b_1}\leq {b_2}\leq \cdots \leq {b_\ell}\leq n$) can be written in $\mathbb{Q}[x_1,\dots, x_n,t]/I$ as a $\mathbb{Q}[t]$-linear combination of the $x_T$ where $T$ is a standard tableau. 
We prove this by induction on $\ell$. The base case $\ell=0$ is clear. Now we assume that $\ell\geq 1$ and the claim holds for $\ell-1$.
The relations \eqref{eq:3.9} imply that 
\begin{equation} \label{eq:3.11}
x_i^2=(n-k+i+1)tx_i+t\sum_{1\leq p\leq i-1} x_p-\sum_{1\leq p\leq i} (n-k+p)t^2 \ \ \ (1\leq i\leq n)
\end{equation}
by an inductive argument on $i$, so we may assume $b_1<b_2<\cdots <b_\ell$. 

To prove the claim for $\ell$, we consider two cases: $1\leq \ell\leq k$ and $\ell\geq k+1$.

(Case i). Suppose $1\leq \ell\leq k$. We write 
$x_{b_1}x_{b_2}\cdots x_{b_\ell}=x_U$ where 
\begin{center}
$U=$ \begin{picture}(115,15)     
                      \put(0,-7){\framebox(10,10){$b_1$}}                      
                      \put(0,3){\framebox(10,10){$a_1$}}
                      \put(10,-7){\framebox(20,10){$\dots$}}
                      \put(10,3){\framebox(20,10){$\dots$}}
                      \put(30,-7){\framebox(10,10){$b_{\ell}$}}
                      \put(30,3){\framebox(10,10){$a_{\ell}$}}
                      \put(40,3){\framebox(25,10){$a_{{\ell}+1}$}}
                      \put(65,3){\framebox(20,10){$\dots$}}
                      \put(85,3){\framebox(25,10){$a_{n-{\ell}}$}}

\end{picture} 
\end{center}
is a permissible filling of $(n-\ell,\ell)$. 
Let $j$ be the minimal positive integer in the set $\{r \mid a_r>b_r, 1\leq r\leq \ell \}$, i.e., 
\begin{align}
&a_i<b_i \ \ \ \ \ \ (1\leq i<j),    \label{eq:3.12} \\
&a_j>b_j.  \label{eq:3.13}
\end{align}
We consider the following equation which follows from the relation \eqref{eq:3.8}:
\begin{align} \label{eq:3.14}
&(-x_{a_1}-x_{a_2}-\dots-x_{a_{j-1}})^j\cdot x_{b_{j+1}}\cdots x_{b_{\ell}} \\ 
&=(x_{b_1}+x_{b_2}+\dots+x_{b_{\ell}}+x_{a_j}+x_{a_{j+1}}+\dots+x_{a_{n-\ell}}-\frac{n(n+1)}{2}t)^j\cdot x_{b_{j+1}}\cdots x_{b_{\ell}}.  \notag
\end{align}

\smallskip
\noindent
{\bf Claim 1.} The left hand side in \eqref{eq:3.14} is a $\mathbb{Q}[t]$-linear combination of the $x_T$ where the $T$ are standard tableaux. 

\smallskip
\noindent
Proof. We expand the left hand side in \eqref{eq:3.14}. Then any monomial which appears in the expansion is of the form 
\begin{equation*} 
x_{a_{{1}}}^{\alpha _1}\cdots x_{a_{{j-1}}}^{\alpha _{j-1}}x_{b_{j+1}}\cdots x_{b_\ell} 
\end{equation*}
where $\sum_{i=1}^{j-1}\alpha_i=j$ and $\alpha_i\ge 0$. 
Note that $\alpha _i>1$ for some $i$ since $\sum_{i=1}^{j-1}\alpha_i=j$ and $\alpha_i\ge 0$.  Therefore, using the relations \eqref{eq:3.11}, the monomial above turns into a sum of elements of the form 
$$f(t)\cdot x_{c_{1}}\cdots x_{c_h}$$ 
where $h<\ell$, $1\leq c_1<\dots <c_h\leq n$, and $f(t)\in \mathbb{Q}[t]$, and by the induction assumption the term above can be written as a $\mathbb{Q}[t]$-linear combination of the $x_T$ where $T$ is a standard tableau.  This proves Claim 1.  \qed

\smallskip
\noindent
{\bf Claim 2.} The right hand side in \eqref{eq:3.14} can be written as a $\mathbb{Q}[t]$-linear combination of $x_U$ and monomials $x_T$ and $x_{U^{\prime}}$ where the coefficient of $x_U$ is equal to $1$, $T$ is a standard tableau on shape $(n-\ell,\ell)$ and $U^{\prime}$ is a permissible filling of $(n-\ell,\ell)$ such that each of the leftmost $j$ columns are strictly increasing (i.e. $a_r<b_r, 1\leq r\leq j$). 

\smallskip
\noindent
Proof. We expand the right hand side in \eqref{eq:3.14}. A monomial which appears in this expansion is of the form 
$$x_{b_{p_1}}^{\beta _1}\cdots x_{b_{p_m}}^{\beta _m}x_{a_{q_{1}}}^{\alpha _1}\cdots x_{a_{q_{h}}}^{\alpha _{h}}x_{b_{j+1}}\cdots x_{b_\ell}$$ 
where $\sum_{i=1}^m\beta_i+\sum_{i=1}^h\alpha_i\leq j,\ \beta_i\ge 1,\ \alpha_i\ge 1$ and $1\leq p_1<\dots<p_m\leq \ell$, $j\leq q_1<\dots<q_{h}\leq n-\ell$.
It is enough to consider the case $\sum_{i=1}^m\beta_i+\sum_{i=1}^h\alpha_i=j$ since if $\sum_{i=1}^m\beta_i+\sum_{i=1}^h\alpha_i<j$ then it follows from the induction assumption that the above form can be written as a $\mathbb{Q}[t]$-linear combination of the $x_T$ where $T$ is a standard tableau. 
If $p_m\ge j+1$ or some $\beta _i$ or $\alpha _i$ is more than $1$, then it follows from the relations \eqref{eq:3.11} and the induction assumption that the monomial above can be written as a linear combination of $x_T$'s over $\mathbb{Q}[t]$ where $T$ is a standard tableau.  If $p_m\le j$ and all $\beta_i$ and $\alpha_i$ are equal to 1, then $h=j-m$ and the monomial above is of the form 
$$x_{b_{p_1}}\cdots x_{b_{p_m}}x_{a_{q_{1}}}\cdots x_{a_{q_{j-m}}}x_{b_{j+1}}\cdots x_{b_\ell}$$ 
where $1\leq p_1<\dots<p_m\leq j\leq q_1<\dots<q_{j-m}\leq n-\ell$.
This monomial is associated to a permissible filling $U^{\prime}$ given by
\begin{center}
$U^{\prime}=$ \begin{picture}(115,15)     
                      \put(0,-7){\framebox(10,10){$d_1$}}                      
                      \put(0,3){\framebox(10,10){$c_1$}}
                      \put(10,-7){\framebox(20,10){$\dots$}}
                      \put(10,3){\framebox(20,10){$\dots$}}
                      \put(30,-7){\framebox(10,10){$d_{\ell}$}}
                      \put(30,3){\framebox(10,10){$c_{\ell}$}}
                      \put(40,3){\framebox(25,10){$c_{{\ell}+1}$}}
                      \put(65,3){\framebox(20,10){$\dots$}}
                      \put(85,3){\framebox(25,10){$c_{n-{\ell}}$}}

\end{picture} 
\end{center}
where  
\begin{align*}
d_i=\begin{cases}
b_{p_i}       \ \ \ \ \ \ \ \ \ \ \ \ \ \ \ \ \ \ \ \ \ \ \ \ \ \ \ \ \ \ \ \ \ \ \ \ \ \ \ \ \ \ \ \ \ \ \ \ \ \ \ \ \ \ \ \ \ \ \ \ \ \ \ \ \ \ \ \ \ \ \ \ \ \ \ \ $if$ \ 1\leq i\leq m,  \\
\text{min} \{ \{{a_{q_{1}}},\cdots ,{{a_{q_{j-m}}},{b_{j+1}},\cdots ,{b_\ell} \}-\{d_{m+1}, \dots ,d_{i-1} \}} \} \ \ \ \ \ \ \ \ $if$ \ m<i\leq \ell, 
\end{cases} 
\end{align*}
and
$$c_i=\text{min} \{ \{{a_{1}},\cdots ,{a_{n-\ell}},{b_{1}},\cdots ,{b_j} \}-\{a_{q_1}, \cdots ,a_{q_{j-m}}, b_{p_1},\cdots ,b_{p_m}, c_{1},\cdots ,c_{i-1} \} \}$$
for $1\leq i\leq n-\ell$.
Note that $x_{U^{\prime}}=x_U$ if and only if $m=j$, since $m=j \Leftrightarrow d_i=b_i$ for $1\leq i\leq \ell$.
We consider the case $m<j$. Since $j\leq q_1$ and $a_j>b_j$ by \eqref{eq:3.13}, we have 
$$c_i=\text{min} \{ \{{a_{1}},\cdots ,{a_{j-1}},{b_{1}},\cdots ,{b_j} \}-\{b_{p_1},\cdots ,b_{p_m}, c_{1},\cdots ,c_{i-1} \} \}$$
for $1\leq i\leq j$.
If $1\leq i\leq m$, we have $c_i\leq a_i<b_i\leq b_{p_i}=d_i$.
If $m<i\leq j$, we have $c_i\leq \text{max} \{a_{j-1}, b_j \}<\text{min} \{a_{j}, b_{j+1} \}\leq d_i$ by \eqref{eq:3.12}, \eqref{eq:3.13}, and $j\leq q_1$.
Thus, $U^{\prime}$ is a permissible filling of $(n-\ell,\ell)$ such that each of the leftmost $j$ columns are strictly increasing (i.e. $a_r<b_r, 1\leq r\leq j$). This proves Claim 2.  \qed

\smallskip

Claims 1 and 2 show that $x_U$ can be written as a $\mathbb{Q}[t]$-linear combination of $x_{U^{\prime}}$ and $x_T$, where $U^{\prime}$ and $T$ are as above. Applying the above discussion for $x_{U^{\prime}}$ in place of $x_{U}$, we see that $x_{U^{\prime}}$ can be written as a $\mathbb{Q}[t]$-linear combination of $x_{U^{\prime \prime}}$ and $x_T$ where $U^{\prime \prime}$ is a permissible filling of $(n-\ell,\ell)$ such that each of the leftmost $j+1$ columns are strictly increasing (i.e. $a_r<b_r, 1\leq r\leq j+1$) and $T$ is a standard tableau.  Repeating this procedure, we can finally express $x_U$ as a $\mathbb{Q}[t]$-linear combination of the $x_T$ where $T$ is a standard tableau. 

(Case ii). If $\ell\geq k+1$, it follows from the relations \eqref{eq:3.10} and the induction assumption that $x_{b_1}x_{b_2}\cdots x_{b_\ell}$  can be 
expressed as a $\mathbb{Q}[t]$-linear combination of the $x_T$ where $T$ is a standard tableau.  

This completes the induction step and proves the proposition.  
\end{proof}

Recall that for a box $b$ in the $i$th row and $j$th column of a Young diagram $\lambda $, $h(i,j)$ denote the number of boxes in the hook formed by the boxes below $b$ in the $j$th column, the boxes to the right of $b$ in the $i$th row, and $b$ itself.

\begin{ex}
For the Young diagram \begin{picture}(40,30)\definecolor{mycolor}{gray}{.2}
                      \put(0,-20){\framebox(10,10)}
                      \put(0,-10){\framebox(10,10)}
                      \put(0,0){\framebox(10,10)}
                      \put(0,10){\framebox(10,10) { \colorbox{mycolor} {} }}   
                      \put(0,20){\framebox(10,10)}                   
                      
                      \put(10,0){\framebox(10,10)}
                      \put(10,10){\framebox(10,10)}
                      \put(10,20){\framebox(10,10)}
                      
                      \put(20,10){\framebox(10,10)}
                      \put(20,20){\framebox(10,10)}
                      
                      \put(30,20){\framebox(10,10)}
                                                         
\end{picture} and the box in the $(2,1)$ location, the hook is \begin{picture}(40,55)\definecolor{mycolor}{gray}{.2}
                      \put(0,-20){\framebox(10,10) { \colorbox{mycolor} {} }}
                      \put(0,-10){\framebox(10,10) { \colorbox{mycolor} {} }}
                      \put(0,0){\framebox(10,10) { \colorbox{mycolor} {} }}
                      \put(0,10){\framebox(10,10) { \colorbox{mycolor} {} }}   
                      \put(0,20){\framebox(10,10)}                   
                      
                      \put(10,0){\framebox(10,10)}
                      \put(10,10){\framebox(10,10) { \colorbox{mycolor} {} }}
                      \put(10,20){\framebox(10,10)}
                      
                      \put(20,10){\framebox(10,10) { \colorbox{mycolor} {} }}
                      \put(20,20){\framebox(10,10)}
                      
                      \put(30,20){\framebox(10,10)}
                                                         
\end{picture} and $h(2,1)=6$.
\end{ex}

\begin{lem}\label{lem4.2} 
Let $\lambda $ be a Young diagram. Let $f^{\lambda }$ denote the number of standard tableaux on $\lambda $.
Then 
$${\begin{pmatrix} n \\ k \end{pmatrix}}=\displaystyle\sum_{0\leq \ell\leq k} f^{(n-\ell,\ell)}.$$
\end{lem}

\begin{proof} We prove the lemma by induction on $k$. As the case $k=0$ is clear, we assume that $k\geq 1$ and that the lemma holds for $k-1$. We use the following hook length formula:
$$f^{\lambda }=\frac{n!}{\displaystyle \Pi_{(i,j)\in \lambda }^{} h(i,j)}.$$
Using the induction assumption and the hook length formula, we have 
\begin{align*}
\displaystyle\sum_{0\leq \ell\leq k} f^{(n-\ell,\ell)} &=\displaystyle\sum_{0\leq \ell\leq k-1} f^{(n-\ell,\ell)}+f^{(n-k,k)} \\
&={\begin{pmatrix} n \\ k-1 \end{pmatrix}}+\frac{n!(n-2k+1)}{(n-k+1)!k!} \\
&={\begin{pmatrix} n \\ k \end{pmatrix}}.
\end{align*}
This completes the induction step and proves the lemma. 
\end{proof}

It follows from Proposition~\ref{prop4.1} and Lemma~\ref{lem4.2} that 
\begin{center}
$ $rank$_{\mathbb{Q}[t]}\mathbb{Q}[x_1,\dots,x_n,t]/I\leq \displaystyle\sum_{0\leq \ell \leq k}f^{(n-\ell,\ell)}={\begin{pmatrix} n \\ k \end{pmatrix}}.$
\end{center}
On the other hand, since the odd degree cohomology groups of $\mathcal{S}_N$ vanish, we have an isomorphism
$H^{\ast}_{S}(\mathcal{S}_N;\mathbb{Q})\cong \mathbb{Q}[t]\otimes H^{\ast}(\mathcal{S}_N;\mathbb{Q})$ as $\mathbb{Q}[t]$-modules, and 
the cellular decomposition of $\mathcal{S}_N$ given by Spaltenstein \cite{spa} (cf. also Hotta-Springer \cite{h-s}) implies that  
$$\dim H^{\ast}(\mathcal{S}_N;\mathbb{Q})
={\begin{pmatrix} n \\ \lambda _N \end{pmatrix}}:={\begin{pmatrix} n \\ \lambda _1!\lambda _2!\cdots \lambda _r! \end{pmatrix}}$$ 
where $\lambda _N=(\lambda _1,\lambda _2,\dots ,\lambda _r)$.
These show 
\begin{center}
$ $rank$_{\mathbb{Q}[t]}  H^{\ast}_{S}(\mathcal{S}_{(n-k,k)};\mathbb{Q}) =\dim_{\mathbb{Q}}  H^{\ast}(\mathcal{S}_{(n-k,k)};\mathbb{Q})
={\begin{pmatrix} n \\ k \end{pmatrix}}. $
\end{center}
Therefore, we have 
\begin{center}
$ $rank$_{\mathbb{Q}[t]}\mathbb{Q}[x_1,\dots,x_n,t]/I\leq $rank$_{\mathbb{Q}[t]}  H^{\ast}_{S}(\mathcal{S}_{(n-k,k)};\mathbb{Q}).$
\end{center}
This means that the epimorphism $\varphi$ in \eqref{eq:3.7} 
is actually an isomorphism, proving Theorem~\ref{theo:3.1}.

\end{document}